\documentclass[12pt]{article}
\font\smallit=cmti10

\usepackage{amsmath,amsthm,amssymb,mathrsfs}
\usepackage{amsfonts}
\usepackage{latexsym}
\usepackage{amscd}

\usepackage{dsfont}
\usepackage{bbm}
\usepackage{rotating}
\usepackage{newlfont}
\usepackage{enumitem}
\usepackage[english]{babel}

\usepackage[english]{babel}

\newtheorem{one}{Theorem}[section]
 \newtheorem{two}[one]{Theorem}
 \newtheorem{three}[one]{Theorem}
 \newtheorem{Lemma 1}[one]{Lemma}
 \newtheorem{Lemma 4}[one]{Lemma}
 \newtheorem{Lemma 2}[one]{Lemma}
 \newtheorem{Lemma 3}[one]{Lemma}
 \newtheorem{four}[one]{Lemma}
  \newtheorem{abc}[one]{Lemma}
  \newtheorem{five}[one]{Lemma}
  \newtheorem{Lemma 5}[one]{Lemma}
  \newtheorem{new}[one]{Lemma}
  \newtheorem{Corollary}[one]{Corollary}

\def\h25{\hspace{-.3cm}}   
\begin{document}
\begin{center}
{\bf On Newman's phenomenon in higher bases}
\vskip 20pt {\bf Sai Teja Somu}\\
{\smallit Department of Mathematics, Indian Institute of Technology
Roorkee, Roorkee - 247 667, India}\\
{\tt somuteja@gmail.com }\\
\end{center}
\vskip 30pt

\centerline{\bf Abstract} A well known result of Newman says that upto a limit, multiples of $3$ with even number of 1's in binary representation always exceed multiples of $3$ with odd number of 1's. The phenomenon of preponderance of even number of 1's is now known as Newman's phenomenon. We show that this phenomenon exists for higher bases.
Let $b$ be a positive integer($\geq 2$). Let $A_{b}$ be the set of all natural numbers which contain only 0's and 1's in b-ary expansion and $S^{(b)}_{q,i}(n)$ be the difference between the corresponding number of $k_e<n$, $k_e\equiv i \mod q$, $k_e\in A_{b}$ and $k_e$ has even number of 1's in b-ary expansion and the number of $k_o$ $k_o<n$, $k_o\equiv i \mod q$, $k_o\in A_{b}$ and $k_o$ has odd number of 1's in b-ary expansion. Let $q$ be a multiple or divisor of $b+1$ which is relatively prime to $b$ then we show that  $S^{(b)}_{q,0}(n)>0$ for sufficiently large $n$. We show that there is a stronger Newman's phenomenon in $A_b$ in the following sense. If $b>2$ and  $n=\sum_{i=0}^{k-1}b_i2^i$ with $b_i\in \{0,1\}$, let $b(n)=\sum_{i=0}^{k-1}b_ib^i$ then  $\lim_{n\rightarrow \infty} \frac{S^{(2)}_{3,0}(n)}{S^{(b)}_{b+1,0}(b(n))}=0$. That is, for the same number of terms there is stronger preponderance in $A_b$ than in $A_2=\mathbb{N}$. In the last section we show that number of primes $p\leq x$ for which $S_{p,0}^{(b)}(n)>0$ for sufficiently large $n$ is $o\left(\frac{x}{\log x}\right)$.

\section{Introduction}

L.Moser conjectured that when the multiples of $3$ are written in binary, upto a limit, numbers with even number of $1$'s always exceed numbers with odd number of $1$'s. 
Newman in \cite{A} proved that the  conjecture is true. To be precise, he proved that $S^{(2)}_{3,0}(n)>0$ for all $n$. If we consider a number $n$ let $b_{k-1}\cdots b_0$ be the binary expansion of $n$. Since $2\equiv -1 \mod 3$, we have \[n\equiv b_0-b_1+b_2\cdots+(-1)^{k-1}b_{k-1} \mod 3.\]
Let $N$ be a natural number and $N=n_{k-1}n_{k-2}\cdots n_0$ be binary expansion of $N$. The difference of multiples of 3 having even number of 1's and odd number of 1's is 
\begin{equation}
S^{(2)}_{3,0}(n)=\sum_{\substack{b_0-b_1+\cdots+(-1)^{k-1}b_{k-1} \equiv 0\mod 3\\b_{k-1}b_{k-2}\cdots b_0<n_{k-1}n_{k-2}\cdots n_0}}(-1)^{b_0+\cdots+b_{k-1}},
\end{equation}
where $<$ denotes lexicographic order among sequences of finite length.
In (1) if we replace $2$ by an even natural number $b$ then the sum appears to increase as $b$ increases. That is, if we consider the following sum 
\begin{equation}
S^{(b)}_{b+1,0}(b(n))=\sum_{\substack{b_0-b_1+\cdots+(-1)^{k-1}b_{k-1} \equiv 0\mod b\\b_{k-1}b_{k-2}\cdots b_0<n_{k-1}n_{k-2}\cdots n_0}}(-1)^{b_0+\cdots+b_{k-1}}
\end{equation}
then the sum seems to increase as $b$ increases.  
So we can guess that $S^{(b)}_{b+1,0}(b(n))>0$  
for sufficiently large $n$. We prove this result in section 2.
Let $A_{b}$ be the set of all natural numbers which contain only 0's and 1's in b-ary expansion. Then upto a limit, multiples of $b+1$ with even number of 1's in $A_b$ always exceed multiples of $b+1$ with odd number of 1's in $A_b$. So there is Newman's phenomenon among the multiples of $b+1$ in the set $A_{b}$. In fact the Newman's phenomenon gets stronger as $b$ increases in the following sense. We show that for any two even numbers $b_1,b_2$ with $b_1>b_2$ we have \[\lim_{n\rightarrow \infty}\frac{S^{(b_2)}_{b_2+1,0}(b_2(n))}{S^{(b_1)}_{b_1+1,0}(b_1(n))}=0.\] That is, for the same number of terms in the sets $A_{b_1}$ and $A_{b_2}$  there is a stronger preponderance in $A_{b_1}$ compared to $A_{b_2}$.
 
 In fact we prove more than $S^{(b)}_{b+1,0}(b(n))>0$. We prove the following theorem in section 2.
 
 \begin{one}\label{Theorem 1}
  If $b$ and $q$ are two positive integers such that $b\geq 2$, $(b+1)|q$ and $(b,q)=1$ where $(b,q)$ is the greatest common divisor of $q$ and $b$. Let $v$ be an integer, then for sufficiently large $N$
  \begin{enumerate}
  \item $S^{(b)}_{q,v(b+1)}(N)>0$.
  \item $S^{(b)}_{q,v(b+1)+1}(N)<0$.
  \item if $b\leq 3$ then $S^{(b)}_{q,v(b+1)-1}(N)<0$.
  \end{enumerate}
 \end{one}
 Theorem \ref{Theorem 1} partially generalizes Theorem 1 of \cite{H} which states that $S^{(2)}_{3k,0}(n)>0$ for almost all $n$.
 
 In section 3 we prove the following result.
 \begin{two}\label{Theorem 2}
   If $d>1$ is a divisor of $(b+1)$ for $b\geq 2$, then for sufficiently large $N$
   \begin{enumerate}
   \item $S^{(b)}_{d,0}(N)>0$.
   \item $S^{(b)}_{d,1}(N)<0$.
   \item $S^{(b)}_{d,-1}(N)\leq0$ and if $d>3$ then $S^{(b)}_{d,-1}(N)<0$.
   \end{enumerate}
  \end{two}
  Let $\mathbb{P}_t(b)$ denote the set of all primes such that $\frac{p-1}{ord_b(p)}=t$ where $ord_p(b)$ denotes the order of $b$ in the multiplicative group $(\mathbb{Z}/p\mathbb{Z})^*$. 
  In the last section we prove that the primes satisfying $S^{(b)}_{p,0}(N)>0$ are of zero density in the set of primes. 
  \begin{three}\label{Theorem 3}
    The number of primes $p \in \mathbb{P}_{t}(b)$ for a given $t>1$ such that $S^{(b)}_{p,0}(N)>0$ for sufficiently large $N$ are bounded by 
       \begin{equation*}
       p\leq C^*t^{2}log^{2}t,
       \end{equation*}
    where $C*>0$  only depends on $b$. If $t=1$ the number of primes are bounded by \[p\leq C_1,\] where $C_1>0$ only depends on $b$. Furthermore,the total number of primes $p\leq x$ such that $S^{(b)}_{p,0}(n)>0$ for sufficiently large $n$ is $o\left(\frac{x}{\log x}\right)$ as $x\rightarrow \infty$. 
   \end{three}
   Theorem 1.3 is a generalization of Theorem 2 of \cite{H}. 
\section{Proof of Theorem \ref{Theorem 1}}
If $b$ is an odd number as $q$ is divisible by $b+1$, $q$ will be an
 even integer and all the integers satisfying the congruences $n\equiv v(b + 1)
\mod q$, $n\equiv v(b + 1) + 1 \mod q$ and $n\equiv v(b + 1)-1 \mod q$ are even, odd and odd
respectively. Also  $n\equiv s_b(n) \mod 2,$ where $s_b(n)$ denotes sum of digits in b-ary expansion. Hence all the numbers
satisfying the congruences $n\equiv v(b + 1) \mod q$, $n\equiv v(b + 1) + 1 \mod q$ and $n\equiv
v(b + 1)\equiv 1 \mod q$ will have even b-ary digit sum, odd b-ary digit sum and odd
b-ary digit sum respectively so (1),(2) and (3) of Theorem \ref{Theorem 1} are trivially true
when $b > 1$ is an odd number. Hence we can assume that $b$ is even.
We prove six lemmas in order to prove Theorem \ref{Theorem 1}.
\begin{Lemma 1}\label{lemma1}
If $N<b^{k}$, \[
S^{(b)}_{q,i}(b^{k}+N)=S^{(b)}_{q,i}(b^{k})-S^{(b)}_{q,i-b^{k}}(N).\]
\end{Lemma 1}
\begin{proof}
Note that if $n<b^k$ then $s_{b}(n+b^k)=s_{b}(n)+1$. Thus, we have
\begin{align*}
S^{(b)}_{q,i}(b^{k}+n)&=\sum_{\substack{n\in A_{b}\\0\leq n<N+b^k\\n\equiv i \mod q}}(-1)^{s_{b}(n)}\\
&=\sum_{\substack{n\in A_{b}\\0\leq n<b^k\\n\equiv i \mod q}}(-1)^{s_{b}(n)}+\sum_{\substack{n\in A_{b}\\0\leq n<N\\n\equiv i-b^k \mod q}}(-1)^{s_{b}(n+b^k)}\\
&=\sum_{\substack{n\in A_{b}\\0\leq n<b^k\\n\equiv i \mod q}}(-1)^{s_{b}(n)}-\sum_{\substack{n\in A_{b}\\0\leq n<N\\n\equiv i-b^k \mod q}}(-1)^{s_{b}(n)}\\
&=S^{(b)}_{q,i}(b^{k})-S^{(b)}_{q,i-b^{k}}(N).
\end{align*}
\end{proof}

From Lemma \ref{lemma1} it follows that if $ N=b^{k_{1}}+\cdots+b^{k_{r}}$ and $b^{k_1}>\cdots >b^{k_r}$ then 
\begin{equation}\label{equation 1}
S^{(b)}_{q,i}(N)=S^{(b)}_{q,i}(b^{k_{1}})-S^{(b)}_{q,i-b^{k_{1}}}(b^{k_{2}})-\cdots+(-1)^{r-1}S^{(b)}_{q,i-b^{k_{1}}-\cdots-b^{k_{r-1}}}(-1,b^{k_{r}}).
\end{equation}
\begin{Lemma 2}\label{Lemma 2}
If $k$ is a natural number then
\[S^{(b)}_{q,i}(b^{k})=\frac{1}{q}\sum\limits_{m=0}^{q-1}\zeta_{q}^{-im}\prod\limits_{p=0}^{k-1}(1-\zeta_{q}^{mb^{p}}).\]
\end{Lemma 2}
\begin{proof}
If $k$ is a natural number, we have
\begin{align*}
S^{(b)}_{q,i}(b^{k})&=\sum_{\substack{n\in A_{b}\\0\leq n<b^k\\n\equiv i \mod q}}(-1)^{s_{b}(n)}\\
&=\frac{1}{q}\sum_{\substack{n\in A_b\\0\leq n<b^k}}\sum_{m=0}^{q-1}(-1)^{s_b(n)}\zeta_q^{(n-i)m}\\
&=\frac{1}{q}\sum_{0\leq\epsilon_0,\cdots,\epsilon_{k-1}\leq 1}\sum_{m=0}^{q-1}(-1)^{\epsilon_0+\cdots+\epsilon_{k-1}}\zeta_q^{(\epsilon_0+\epsilon_1b+\cdots+\epsilon_{k-1}b^{k-1}-i)m}\\
&=\frac{1}{q}\sum\limits_{m=0}^{q-1}\zeta_{q}^{-im}\prod\limits_{p=0}^{k-1}(1-\zeta_{q}^{mb^{p}}).
\end{align*}
\end{proof}

\begin{Lemma 4}\label{Lemma 4}
Let $b=2l$. If $|z|=1$ and $|1-z|>2\sin \frac{\pi l}{2l+1}$ , then \begin{equation*}
|1-z||1-z^{b}|<4\sin^{2}\frac{\pi l}{2l+1}.
\end{equation*} 
\end{Lemma 4}
\begin{proof}
Let $z$ be $e^{i\theta}$, where $\theta\in [0,2\pi]$. As $|1-z|>2\sin \frac{\pi l}{2l+1}$ we have $\frac{2\pi l}{2l+1}<\theta<\frac{2\pi(l+1)}{2l+1}.$  Let  \begin{equation*}f(\theta):=4|\sin l\theta \sin \frac{\theta}{2}|=|1-z||1-z^{b}|.\end{equation*} If $\frac{2\pi l}{2l+1}<\theta<\frac{2\pi(l+1)}{2l+1}$, it can be seen that $|tan l\theta|\leq |tan \frac{\pi l}{2l+1}|\leq |\tan \frac{\theta}{2}|.$

 We have  \begin{align*}
  & f(\theta)=4\delta \sin l\theta \sin \frac{\theta}{2}, \\
  & f'(\theta) =4\delta (l\cos l\theta sin\frac{\theta}{2}+\frac{1}{2}\sin l\theta \cos \frac{\theta}{2} ).
 \end{align*}
 
\begin{center}
    \begin{tabular}{| l | l |l | l | l| l | l | l|}
    \hline
     & $\theta$& $\sin l\theta$ & $\cos l\theta $&$\sin \frac{\theta}{2}$ & $\cos \frac{\theta}{2}$ & $\delta$ &$f'(\theta)$ \\ \hline
    $l=2m$ & $\frac{2l\pi}{2l+1}<\theta\leq \pi$  & $\leq 0$&$\geq 0$& $>0$& $>0$& $-1$& $<0$
     \\ \hline
     $l=2m$& $\pi<\theta<\frac{(2l+2)\pi}{2l+1}$  & $>0$& $>0$& $>0$ & $<0$ & $+1$& $>0$ \\ \hline
    $l=2m+1$ &$\frac{2l\pi}{2l+1}<\theta\leq \pi$  & $>0$ &$<0$ &$>0$ &$>0$ &$+1$ &$<0$ \\
    \hline
    $l=2m+1$& $\pi <\theta<\frac{(2l+2)\pi}{2l+1}$ & $<0$ & $<0$&$>0$ &$<0$ &$-1$ &$>0$ \\
    \hline
    \end{tabular}
\end{center}
So, in any interval 
\begin{equation*}
      f(\theta)<f(\pi \pm \frac{\pi}{2l+1})=4\sin^2\frac{\pi l}{2l+1}. 
\end{equation*}
\end{proof}
Let $s$ be order of $b$ in the multiplicative group $(\mathbb{Z}/q\mathbb{Z})^*$.
\begin{four}\label{Lemma 6}
If $(q,b)=1$, $(b+1)|q$ and  $b$ is even, $m\equiv  \frac{\pm ql}{2l+1} \mod q$ then \[\left|\prod_{p=0}^{s-1}(1-\zeta_q^{b^{p}m})\right|=\left(2\sin\frac{\pi l}{2l+1}\right)^s,\] and if  $m\not\equiv \frac{\pm ql}{2l+1} (\mod q)$ then \[\left|\prod_{p=0}^{s-1}(1-\zeta_q^{b^{p}m})\right|<\left(2\sin\frac{\pi l}{2l+1}\right)^s.\]
\end{four}
\begin{proof}
Observe that $|1-\zeta_q^{mb^p}|=2\sin \frac{\pi l}{2l+1}$ for $0\leq p \leq s-1$ if and only if $m\equiv \frac{\pm ql}{2l+1} \mod q$. Hence \[\left|\prod_{p=0}^{s-1}(1-\zeta_q^{b^{p}m})\right|=\left(2\sin\frac{\pi l}{2l+1}\right)^s\] when $m\equiv \frac{\pm ql}{2l+1} \mod q$. If $m\not\equiv \frac{\pm ql}{2l+1} (\mod q)$ then $|1-\zeta_q^{mb^i}|\neq 2\sin \frac{\pi l}{2l+1}$ for $0\leq i \leq s-1$. Let $S_1$,$S_2$ and $S_3$ be subsets of $\{0,1,\cdots,s-1\}$. $S_1$ contains all the $p$ such that $|1-\zeta_q^{mb^p}|>2\sin \frac{\pi l}{2l+1}$, $S_2$ all $p$ such that $p-1(\mod s)\in S_1$ and $S_3$ contains the remaining elements of $\{0,1,\cdots,s-1\}$. Clearly, from Lemma \ref{Lemma 4} $S_1\cap S_2=\phi$ and if $p\in S_1$ then   $|(1-\zeta_q^{b^{p}m})(1-\zeta_q^{b^{p+1}m})|<(2\sin \frac{\pi l}{2l+1})^2$.
Therefore 
\begin{align*}
\left|\prod_{p=0}^{s-1}(1-\zeta_q^{b^{p}m})\right|&=\prod_{p\in S_1}|(1-\zeta_q^{b^{p}m})(1-\zeta_q^{b^{p+1}m})|\prod_{p\in S_3}|(1-\zeta_q^{b^{p}m})|\\
&<\left(2\sin\frac{\pi l}{2l+1}\right)^s.
\end{align*}
\end{proof}
Let $\gamma=max\{\left|\prod_{p=0}^{s-1}(1-\zeta_q^{b^pm})\right|: m \not \equiv \pm \frac{ql}{2l+1}\mod q\}$. From Lemma \ref{Lemma 6} we have $\gamma<\left(2\sin \frac{\pi l}{2l+1}\right)^s$.

\begin{abc}\label{Lemma 9}
\[
S^{(b)}_{q,i}(b^{k})=\begin{cases}
 \frac{2}{q}(\cos\frac{2\pi il}{2l+1})\left(2\sin \frac{\pi l}{2l+1}\right)^{k}+O(\gamma^{\frac{k}{s}})&  k ~ \text{even}
   \\ \frac{2}{q}\left(\cos\frac{2\pi il}{2l+1}-\cos\frac{2\pi(i-1)l}{2l+1}\right)\left(2\sin \frac{\pi l}{2l+1}\right)^{k-1}+O(\gamma^{\frac{k}{s}}) & k ~ \text{odd.} 
   \end{cases}
\]
\end{abc}
\begin{proof}
Let $k=k_1s+k_2$ where $0\leq k_2 \leq s-1$.
From Lemma \ref{Lemma 2} we have
 \begin{equation}\label{equation 2}
 S^{(b)}_{q,i}(b^k)=\frac{1}{q}\sum_{m\equiv \frac{\pm ql}{2l+1}}\zeta_q^{-im}\prod_{p=0}^{k-1}\left(1-\zeta_q^{b^pm}\right)+\frac{1}{q}\sum_{m\not\equiv \frac{\pm ql}{2l+1}}\zeta_q^{-im}\prod_{p=0}^{k-1}\left(1-\zeta_q^{b^pm}\right)
 \end{equation}
where the first term of right hand side of equation (\ref{equation 2}) is
\begin{align*}
\frac{1}{q}\sum_{m\equiv \frac{\pm ql}{2l+1}}\zeta_q^{-im}\prod_{p=0}^{k-1}\left(1-\zeta_q^{b^pm}\right)&=\frac{1}{q}\zeta_{b+1}^{-il}\prod_{p=0}^{k-1}\left(1-\zeta_{b+1}^{lb^p}\right)+\frac{1}{q}\zeta_{b+1}^{+il}\prod_{p=0}^{k-1}\left(1-\zeta_{b+1}^{-lb^p}\right).
\end{align*}
Hence

\[
\frac{1}{q}\sum_{m\equiv \frac{\pm ql}{2l+1}}\zeta_q^{-im}\prod_{p=0}^{k-1}\left(1-\zeta_q^{b^pm}\right)=\begin{cases}

 \frac{2}{q}\cos\frac{2\pi il}{2l+1}\left(2\sin \frac{\pi l}{2l+1}\right)^{k}&  k ~ \text{even}
   \\ \frac{2}{q}\left(\cos\frac{2\pi il}{2l+1}-\cos\frac{2\pi(i-1)l}{2l+1}\right)\left(2\sin \frac{\pi l}{2l+1}\right)^{k-1}& k ~ \text{odd.} 
   \end{cases}
\]
The second term of right hand hand side of equation (\ref{equation 2})
\begin{align*}
\left|\frac{1}{q}\sum_{m\not\equiv \frac{\pm ql}{2l+1}}\zeta_q^{-im}\prod_{p=0}^{k-1}\left(1-\zeta_q^{b^pm}\right)\right|=\left|\frac{1}{q}\sum_{m\not\equiv \frac{\pm ql}{2l+1}}\zeta_q^{-im}\left(\prod_{p=0}^{s-1}(1-\zeta_q^{b^pm})\right)^{k_1}\left(\prod_{p=0}^{k_2-1}(1-\zeta_q^{b^pm})\right)\right|.
\end{align*}
Let $C=max\{\left|\prod_{p=0}^{k_2-1}\left(1-\zeta_q^{b^pm}\right)\right|: m\in \mathbb{N},~ 0\leq k_2 \leq s-1 \}$ then we have 
\[\left|\frac{1}{q}\sum_{m\not\equiv \frac{\pm ql}{2l+1}}\zeta_q^{-im}\prod_{p=0}^{k-1}\left(1-\zeta_q^{b^pm}\right)\right|\leq C\gamma^{k_1}=O(\gamma^{\frac{k}{s}})\]
which implies
\[
S^{(b)}_{q,i}(b^{k})=\begin{cases}
 \frac{2}{q}\cos\frac{2\pi il}{2l+1}\left(2\sin \frac{\pi l}{2l+1}\right)^{k}+O(\gamma^{\frac{k}{s}}),&  k ~ \text{even}
   \\ \frac{2}{q}\left(\cos\frac{2\pi il}{2l+1}-\cos\frac{2\pi(i-1)l}{2l+1}\right)\left(2\sin \frac{\pi l}{2l+1}\right)^{k-1}+O(\gamma^{\frac{k}{s}}) & k ~ \text{odd.} 
   \end{cases}
\]. 
\end{proof}

\begin{five}\label{Lemma 5}
 If $b\geq 4$ is even and $b=2l$ then 
 there exists a constant $M>0$ depending upon $b$,$q$ and not depending on $k$ such that
\begin{enumerate}
\item \begin{equation*}
 S_{q,v(b+1)}^{(b)}(b^{k})\geq \frac{1}{q}\left(2\sin \frac{\pi l}{2l+1}\right)^{k+1}-M\gamma^{\frac{k}{s}}
 \end{equation*} 
\item
 \begin{equation*}
S_{q,v(b+1)+1}^{(b)}(b^{k})\leq \frac{2}{q}\cos \frac{2\pi l}{2l+1}\left(2\sin \frac{\pi l}{2l+1}\right)^{k}+M\gamma^{\frac{k}{s}}
\end{equation*}
\item
\begin{equation*}
S_{q,v(b+1)-1}^{(b)}(b^{k})\leq \frac{2}{q}\left(\cos \frac{2\pi l}{2l+1}-\cos \frac{4\pi l}{2l+1}\right)
\left(2\sin \frac{\pi l}{2l+1}\right)^{k-1}+M\gamma^{\frac{k}{s}}
\end{equation*}
\item
\begin{equation*}
\left|S_{q,i}^{(b)}(b^{k})\right|\leq \frac{2}{q}\left(2\sin\frac {\pi l}{2l+1}\right)^{k}+M\gamma^{\frac{k}{s}}
\end{equation*}
\item
\begin{equation*}
S_{q,v(b+1)\pm 1}^{(b)}(b^{k})\leq M\gamma^{\frac{k}{s}}
\end{equation*}
\item
\begin{equation*}
S_{q,v(b+1)+0\text{ or }\pm2}^{(b)}(b^{k})\geq -M\gamma^{\frac{k}{s}}
\end{equation*}
\end{enumerate}
\end{five}
\begin{proof}
From Lemma \ref{Lemma 9} we have
\begin{align*}
& S^{(b)}_{q,i}(b^{k})=\begin{cases}
\frac{2}{q}(\cos\frac{2\pi il}{2l+1})\left(2\sin \frac{\pi l}{2l+1}\right)^{k}+O(\gamma^{\frac{k}{s}})  & k ~~\text{even}
   \\ \frac{2}{q}\left(\cos\frac{2\pi il}{2l+1}-\cos\frac{2\pi(i-1)l}{2l+1}\right)\left(2\sin \frac{\pi l}{2l+1}\right)^{k-1}+O(\gamma^{\frac{k}{s}}) & k ~~\text{odd}.
   \end{cases} \\\text{Hence for} & ~ i=v(b+1),~v(b+1)+1~\text{and} ~ v(b+1)-1  \\& S_{q,v(b+1)}^{(b)}(b^{k})=
\begin{cases}
 \frac{2}{q}\left(2\sin \frac{\pi l}{2l+1}\right)^k+O(\gamma^{\frac{k}{s}}) ~  &
  k ~ \text{even} 
   \\\frac{1}{q}\left(2\sin \frac{\pi l}{2l+1}\right)^{k+1}+O(\gamma^{\frac{k}{s}}) & k ~ \text{odd}
   \end{cases}\\   & S_{q,v(b+1)+1}^{(b)}(b^{k})=\begin{cases}
    \frac{2}{q}\cos\frac{2\pi l}{2l+1}\left(2\sin \frac{\pi l}{2l+1}\right)^k+O(\gamma^{\frac{k}{s}}) ~ & k ~ \text{even} 
      \\-\frac{1}{q}\left(2\sin \frac{\pi l}{2l+1}\right)^{k+1}+O(\gamma^{\frac{k}{s}}) ~ & k ~ \text{odd} \\
   \end{cases}\\  & S_{q,v(b+1)-1}^{(b)}(b^{k})=\begin{cases} \frac{2}{q}\cos\frac{2\pi l}{2l+1}\left(2\sin \frac{\pi l}{2l+1}\right)^k+O(\gamma^{\frac{k}{s}})~ & k ~\text{even} 
         \\\frac{2}{q}\left({\cos \frac{2\pi l}{2l+1}-\cos\frac{4\pi l}{2l+1}}\right)\left(2\sin \frac{\pi l}{2l+1}\right)^{k-1}+O(\gamma^{\frac{k}{s}}) & k ~\text{odd}.
      \end{cases}
      \end{align*}
      Results (1),(2),(3) and (4) follow from the inequalities
       \begin{align*}
            & 2\sin\frac{\pi l}{2l+1}\leq 2,\\
            & 2\cos\frac{2\pi l}{2l+1}\geq -2\sin\frac{\pi l}{2l+1},\\
            &\cos \frac{2\pi l}{2l+1}-\cos\frac{4\pi l}{2l+1}\geq 2\cos\frac{2\pi l}{2l+1}\sin\frac{\pi l}{2l+1},\\
            &\left|\cos \frac{2\pi l}{2l+1}\right|\leq 1 \text{ and }\\
            &\left|\cos\frac{2\pi il}{2l+1}-\cos\frac{2\pi(i-1)l}{2l+1}\right|\leq 2\sin\frac{\pi l}{2l+1}.
            \end{align*}
      
      If $l\geq 2$ from the inequalities  $ 
    \cos\frac{2\pi l}{2l+1}<0$  and $\cos\frac{4\pi l}{2l+1}>0$ we have
    \begin{equation*}
   S_{q,v(b+1)\pm 1}^{(b)}(b^{k})\leq M\gamma^{\frac{k}{s}},
   \end{equation*}
   for some $M$.
   Hence (5) is true.
  \begin{equation*}
   S_{q,v(b+1)+2}^{(b)}(b^{k})=\begin{cases} \frac{2}{q}\cos\frac{4\pi l}{2l+1}\left(2\sin \frac{\pi l}{2l+1}\right)^k+O(\gamma^{\frac{k}{s}}) ~ & k ~\text{even} 
            \\ \frac{2}{q}\left(\cos \frac{4\pi l}{2l+1}-\cos\frac{2\pi l}{2l+1}\right)\left(2\sin \frac{\pi l}{2l+1}\right)^{k-1}+O(\gamma^{\frac{k}{s}})~& k~\text{odd}
         \end{cases}
         \end{equation*}
   Hence 
   \begin{equation*}S_{q,v(b+1)+2}^{(b)}(b^{k})\geq -M\gamma^{\frac{k}{s}}. \end{equation*} 
  \begin{equation*}
    S_{q,v(b+1)-2}^{(b)}(b^{k})=\begin{cases} \frac{2}{q}\cos\frac{4\pi l}{2l+1}\left(2\sin \frac{\pi l}{2l+1}\right)^k+O(\gamma^{\frac{k}{s}}) ~ & k ~ \text{even} 
               \\\frac{2}{q}\left(2\sin \frac{\pi l}{2l+1}\right)^{k-1}\left(\cos \frac{4\pi l}{2l+1}-\cos\frac{6\pi l}{2l+1}\right)+O(\gamma^{\frac{k}{s}})~ & k~\text{odd}
            \end{cases}\end{equation*}
   As $\cos \frac{4\pi l}{2l+1}-\cos\frac{6\pi l}{2l+1}\geq 0$ for $l\geq 2$ we have
   \begin{equation*}S_{q,v(b+1)-2}^{(b)}(b^{k})\geq -M\gamma^{\frac{k}{s}}\end{equation*} implying (6).
\end{proof}

Proof of Theorem \ref{Theorem 1}
\begin{proof}
Theorem 1.3 of \cite{F} and Theorem 1 of \cite{H} covers the case $b=2$ so we can assume $b\geq 4$.
From (\ref{equation 1}) if
\begin{equation*}
 N=b^{k_{1}}+\cdots+b^{k_{r}}\end{equation*} where $ k_{1}>\cdots>k_{r}$ we have
\[S^{(b)}_{q,i}(N)=S^{(b)}_{q,i}(b^{k_{1}})-S^{(b)}_{q,i-b^{k_{1}}}(b^{k_{2}})+S^{(b)}_{q,i-b^{k_{1}}-b^{k_{2}}}(b^{k_{3}})-\cdots.\] 
When $i=v(b+1)$ we have
\begin{equation*}S^{(b)}_{q,v(b+1)}(N)=S^{(b)}_{q,v(b+1)}(b^{k_{1}})-S^{(b)}_{q,v(b+1)-b^{k_{1}}}(b^{k_{2}})+S^{(b)}_{q,v(b+1)-b^{k_{1}}-b^{k_{2}}}(b^{k_{3}})-\cdots
\end{equation*} 
As $b^{k}\equiv \pm 1 \mod (b+1)$, from Lemma \ref{Lemma 5} we have 
\begin{align*}
 & S_{q,v(b+1)}^{(b)}(b^{k})\geq \frac{1}{q}\left(2\sin \frac{\pi l}{2l+1}\right)^{k+1}-M\gamma^{\frac{k}{s}},  \\ & S^{(b)}_{q,v(b+1)-b^{k_{1}}}(b^{k_{2}})\leq M\gamma^{\frac{k}{s}}\text{ and }\\& \left|S_{q,i}^{(b)}(b^{k})\right|\leq \frac{2}{q}\left(2\sin \frac{\pi l}{2l+1}\right)^{k}+M\gamma^{\frac{k}{s}}.
 \end{align*}
Let $\beta=2\sin\frac{\pi l}{2l+1}$
then \begin{align*}S^{(b)}_{q,v(b+1)}(N)&\geq \left(1+o(1)\right)\left(\frac{1}{q}\left(2\sin \frac{\pi l}{2l+1}\right)^{k_{1}+1}- \frac{2}{q}\left(2\sin \frac{\pi l}{2l+1}\right)^{k_{1}-2}- \frac{2}{q}\left(2\sin \frac{\pi l}{(2l+1)}\right)^{k_{1}-3}+\cdots\right)
\\&=\left(1+o(1)\right)\left(\frac{\beta^{k_{1}+1}}{q}-\frac{2}{q}\frac{\beta^{k_1-1}}{\beta-1}\right).
\end{align*}
Hence (1) of Theorem \ref{Theorem 1} is true.
\begin{equation*}
S^{(b)}_{q,v(b+1)+1}(N)=S^{(b)}_{q,v(b+1)+1}(b^{k_{1}})-S^{(b)}_{q,v(b+1)+1-b^{k_{1}}}(b^{k_{2}})+S^{(b)}_{q,v(b+1)+1-b^{k_{1}}-b^{k_{2}}}(b^{k_{3}})-\cdots \end{equation*}
From (2) of Lemma \ref{Lemma 5} and from (6) of Lemma \ref{Lemma 5}
\begin{equation*}
S^{(b)}_{q,v(b+1)+1}(b^{k_{1}})\leq \frac{2}{q}\cos \frac{2\pi l}{2l+1}\left(2\sin \frac{\pi l}{2l+1}\right)^{k_{1}}+M\gamma^{\frac{k}{s}}
\end{equation*} and 
\begin{equation*}
S_{q,v(b+1)+1-b^{k_{2}}}^{(b)}(b^{k})\geq -M\gamma^{\frac{k}{s}}. 
\end{equation*}
Therefore
\begin{align*}
S^{(b)}_{q,v(b+1)+1}(N)&\leq \left(1+o(1)\right)\left(\frac{2}{q}\cos \frac{2\pi l}{2l+1}\left(2\sin \frac{\pi l}{2l+1}\right)^{k_{1}}+\frac{2}{q}\left(2\sin \frac{\pi l}{2l+1}\right)^{k_{1}-2}+\frac{2}{q}\left(2\sin \frac{\pi l}{2l+1}\right)^{k_{1}-3}\cdots\right)\\&=\left(1+o(1)\right)\left(\frac{2}{q}\cos\frac{2\pi l}{2l+1}\beta^{k_{1}}+\frac{2\beta^{k_{1}-1}}{q(\beta-1)}\right)\\&=(1+o(1))\left(\frac{2}{q}\beta^{k_{1}}\left(-\cos\frac{\pi}{2l+1}+\frac{1}{\beta(\beta-1)}\right)\right). \end{align*}
Hence (2) of Theorem \ref{Theorem 1} is true for sufficiently large $N$.
\begin{equation*}
S^{(b)}_{q,v(b+1)-1}(N)=S^{(b)}_{q,v(b+1)-1}(b^{k_{1}})-S^{(b)}_{q,v(b+1)-1-b^{k_{1}}}(b^{k_{2}})+S^{(b)}_{q,v(b+1)-1-b^{k_{1}}-b^{k_{2}}}(b^{k_{3}})-\cdots \end{equation*}
From (3) of Lemma \ref{Lemma 5} and (6) of Lemma \ref{Lemma 5}
\begin{equation*}
S^{(b)}_{q,v(b+1)-1}(b^{k_{1}})\leq (1+o(1))\frac{2}{q}\left(\cos \frac{2\pi l}{2l+1}-\cos \frac{4\pi l}{2l+1}\right)\left(2\sin \frac{\pi l}{2l+1}\right)^{k_{1}-1}
\end{equation*} and \begin{equation*}S^{(b)}_{q,v(b+1)-1-b^{k_{1}}}(b^{k_{2}})\geq -M\gamma^{\frac{k}{s}}. 
\end{equation*}
Therefore \begin{align*}
S^{(b)}_{q,v(b+1)-1}(N)&\leq (1+o(1))\left(\frac{2}{q}\left(\cos \frac{2\pi l}{2l+1}-\cos \frac{4\pi l}{2l+1}\right)\left(2\sin \frac{\pi l}{2l+1}\right)^{k_{1}-1}+\frac{2}{q}\left(2\sin \frac{\pi l}{2l+1}\right)^{k_{1}-2}+\cdots\right)\\&=(1+o(1))\frac{2}{q}\beta^{k_{1}-1}\left(\cos\frac{2\pi l}{2l+1}-\cos\frac{4\pi l}{2l+1}+\frac{1}{\beta-1}\right)\leq 0. \end{align*}
Hence (3) of Theorem \ref{Theorem 1} is true.
\end{proof}
\begin{Corollary}
If $b_1$ and $b_2$ are two even numbers and $b_1>b_2$ then \[\lim_{n\rightarrow \infty}\frac{S^{(b_2)}_{b_2+1,0}(b_2(n))}{S^{(b_1)}_{b_1+1,0}(b_1(n))}=0.\]
\end{Corollary}
\begin{proof}
Let $n=\sum\limits_{i=0}^{k-1}\epsilon_i2^i$ where $\epsilon_i\in \{0,1\}$.
From Lemma \ref{Lemma 5} and Theorem \ref{Theorem 1} one can prove that,  
for every even $b$ for sufficiently large $n$ there exists constants $c_1>0$ and $c_2>0$ independent of $k$ such that
\begin{equation}\label{equation cor}
c_1\left(2\sin \frac{\pi b}{2b+2}\right)^k<S^{(b)}_{b+1,0}(b(n))<c_2\left(2\sin \frac{\pi b}{2b+2}\right)^k.
\end{equation}
The Corollary follows from (\ref{equation cor}). 
\section{Proof of Theorem \ref{Theorem 2}}
\end{proof}
Proof of Theorem \ref{Theorem 2}
\begin{proof}
If $d$ is even then $b$ is odd and the result is trivial. If $d\geq 3$ is odd and $d|(b+1)$. Let $\phi:A_{d-1}\rightarrow A_{b}$ be a map defined by \begin{equation*}\phi((d-1)^{k_{1}}+\cdots+(d-1)^{k_{r}})=b^{k_{1}}+\cdots+b^{k_{r}} \end{equation*} for $ k_{1}>\cdots>k_{r}$.
It is easy to see that $n\in A_{d-1}\implies n\equiv \phi(n) \mod d.$ 
Hence \begin{equation*}S^{(d-1)}_{d,i}((d-1)^{k_{1}}+\cdots+(d-1)^{k_{r}})=S^{(b)}_{d,i}(b^{k_{1}}+\cdots+b^{k_{r}}).\end{equation*}
From previous theorem for sufficiently large $n$  \begin{equation*}
     S^{(d-1)}_{d,0}(n)> 0,
     S^{(d-1)}_{d,1}(n)< 0 \text{ and } 
     S^{(d-1)}_{d,-1}(n)\leq 0.  \end{equation*}
Hence
\begin{equation*} S^{(b)}_{d,0}(n)>0,S^{(b)}_{d,1}(n)<0 \text{ and }  S^{(b)}_{d,-1}(n)\leq 0. \end{equation*}
\end{proof}
\section{ Proof of Theorem \ref{Theorem 3} }
Proof of Theorem \ref{Theorem 3}
\begin{proof}
The proof will be similar to proof of Theorem 2 in \cite{H} and Theorem 1.8 in \cite{F}. Note that in this proof $b$ need not be even and $b$ need not equal $2l$. Let $s$ be the order of $b$ in $(\mathbb{Z}/p\mathbb{Z})^{*}$ and let $L_{1},\cdots,L_{t}$ be cosets of $\{1,b,\cdots,b^{s-1}\}$. From Lemma \ref{Lemma 2} we have
\begin{equation*}
S_{p,0}^{(b)}(b^{4ks-2})=\frac{1}{p}\sum_{r=1}^{t}\left(\prod\limits_{j=0}^{s-1}(1-\zeta_p^{2mb^j})\right)^{4k}\left(\sum_{l\in L_{r}}\frac{1}{(1-\zeta_{p}^l)(1-\zeta_{p}^{lb})}\right),
\end{equation*}
where $m$ is picked from the set $L_r$. We have 
\begin{align*}
\left(\prod\limits_{j=0}^{s-1}(1-\zeta_p^{2mb^j})\right)^{4k}=\left(\prod_{i=0}^{s-1}\zeta_p^{mb^i}\right)^{4k}\prod\limits_{i=0}^{s-1}\left(\zeta_p^{-mb^i}-\zeta_p^{mb^i}\right)^{4k}\geq 0,
\end{align*}
and
\begin{align*}
Re\left(\sum_{l\in L_{r}}\frac{1}{(1-\zeta_{p}^{l})(1-\zeta_{p}^{lb})}\right)&=\sum_{l\in L_{r}}-\frac{\left(\cos \frac{2\pi lb}{2p} \cos \frac{2\pi l}{2p}-\sin \frac{2\pi lb}{2p} \sin \frac {2\pi l}{2p} \right)}{4\sin \frac{2\pi lb}{2p} \sin  \frac{2\pi l}{2p}}\\
&=-\sum_{l\in L_{r}}\left(-\frac{1}{4}+\frac{1}{4\tan \frac{2\pi lb}{2p}\tan\frac{2\pi l}{2p}}\right)\\
&=\frac{s}{4}-\sum_{l\in L_{r}}\frac{1}{4\tan \frac{\pi lb}{p}\tan\frac{\pi l}{p}}.
\end{align*}
Now the following Lemma will help in completing the proof of Theorem \ref{Theorem 3}.

\begin{Lemma 5}\label{Lemma 8}
Let $L$ be a coset of $\{1,b,\cdots,b^{s-1}\}$ and $p\geq Ct^2(\log p)^2$ then 
\[\sum_{l\in L}\frac{1}{4\tan \frac{\pi lb}{p}\tan \frac{\pi l}{p}}\geq \frac{C_1(b)p^{\frac{3}{2}}}{t^2\log p}-C_2(b)s   \]
for some positive constants $C_1(b)$,$C_2(b)$ and $C$ which only depend on $b$.
\end{Lemma 5}
From Lemma \ref{Lemma 8}
\[ Re\left(\sum_{l\in L_{r}}\frac{1}{(1-\zeta_{p}^{l})(1-\zeta_{p}^{lb})}\right)\leq \frac{-C_1(b)p^{\frac{3}{2}}}{t^2\log p}+\left(C_2(b)+\frac{1}{4}\right)s  \]
Hence if 
\begin{equation*}
p>max\left\{\left(\frac{C_2(b)+\frac{1}{4}}{C_1(b)}\right)^2(t\log p)^2,C(t\log p)^2\right\}\leq C'(t\log p)^2
\end{equation*}
then $S^{(b)}_{p,0}(b^{4ks-2})<0$. Hence if $p\leq C'(t\log p)^2$ and $t>1$ one can prove that there exists a constant $C^*$ such that 
\begin{equation*}
p\leq C^*(t\log t)^2.
\end{equation*}
From the following variation of result of Erd\H{o}s we can prove the second part of the theorem.
 
 \begin{new}
 For every integer $b\geq 2$ and every sequence $\epsilon_{p}\rightarrow 0$ as $p\rightarrow \infty$ we have $\left|\left\{ p\leq x:ord_{p}(b)\leq p^{\frac{1}{2}+\epsilon_{p}}\right \}\right|=o\left(\frac{x}{\log x}\right)$. 
 \end{new}
 \end{proof}
 Proof of Lemma \ref{Lemma 8} 
 \begin{proof}
 The residues are taken from $-\frac{p}{2}$ and $\frac{p}{2}$. If we partition $L$ into four sets $P_1$, $P_2$, $P_3$ and $P_4$. $P_1$ contains all $l\in L$ satisfying $|l|\leq \frac{p}{4b}$ and $|b^{-1}l\mod p|>\frac{p}{4b}$, $P_2$ contains all $l$ satisfying $|l|\leq \frac{p}{4b}$ and $|b^{-1}l\mod p|\leq \frac{p}{4b}$, $P_3$ contains all $l$ satisfying $|l|> \frac{p}{4b}$ and $|bl\mod p|\leq \frac{p}{4b}$  and $P_4$ contains all $l$ satisfying $|l|> \frac{p}{4b}$ and $|bl\mod p|> \frac{p}{4b}$.
 We have
\begin{multline*}
\sum:=\sum_{l\in L}\frac{1}{4\tan \frac{\pi lb}{p}\tan\frac{\pi l}{p}}=\sum_{l\in P_1}\frac{1}{4\tan \frac{\pi lb}{p}\tan\frac{\pi l}{p}}+\sum_{l\in P_2}\frac{1}{4\tan \frac{\pi lb}{p}\tan\frac{\pi l}{p}}\\+ \sum_{l\in P_3}\frac{1}{4\tan \frac{\pi lb}{p}\tan\frac{\pi l}{p}}+\sum_{l\in P_4}\frac{1}{4\tan \frac{\pi lb}{p}\tan\frac{\pi l}{p}}.\end{multline*}

Observe that 
\begin{equation}\label{equation 5}
l\in P_1 \Leftrightarrow b^{-1}l \in P_3.
\end{equation}

So \begin{equation}\label{equation 6}
\sum_{l\in P_1 \cup P_3}\frac{1}{4\tan \frac{\pi lb}{p}\tan\frac{\pi l}{p}}=\sum_{l\in P_1}\left(\frac{1}{4\tan \frac{\pi lb}{p}\tan\frac{\pi l}{p}}+\frac{1}{4\tan \frac{\pi lb^{-1}}{p}\tan\frac{\pi l}{p}}\right).\end{equation}
If $|l|\leq \frac{p}{4b}$ then 
 \begin{equation}\label{equation 7}
 \frac{1}{4\tan\frac{\pi lb}{p}\tan\frac{\pi l}{p}}=\frac{\cos\frac{\pi lb}{p} \cos\frac{\pi l}{p}}{4\sin \frac{\pi lb}{p} \sin\frac{\pi l}{p}}\geq \frac{\cos\frac{\pi}{4}\cos\frac{\pi}{4b}}{\frac{4\pi lb}{p}\frac{\pi l}{p}}=c_1(b)\frac {p^2}{l^2}. 
 \end{equation}
 If $|b^{-1}l mod p|>\frac{p}{4b}$ then 
 $|\tan \frac{\pi b^{-1}l}{p}|\geq \tan\frac{\pi}{4b}$ 
 and 
 $|\tan \frac{\pi l}{p}|\geq \frac{\pi l}{p}$ which implies
 \begin{equation}\label{equation 8}
 \frac{1}{4\tan \frac{\pi lb^{-1}}{p}\tan\frac{\pi l}{p}}\geq \frac{-1}{\left|4\tan \frac{\pi lb^{-1}}{p}\right|\left|\tan\frac{\pi l}{p}\right|}\geq \frac{-p}{4\tan\frac{\pi}{4b}l}=-\frac{c_2(b)|p|}{|l|}.
 \end{equation}
 If $|l|>\frac{p}{4b}$ and $|bl mod p|>\frac{p}{4b}$
 then 
 \begin{equation}\label{equation 9}
 \frac{1}{4\tan \frac{\pi lb^{-1}}{p}\tan\frac{\pi l}{p}}\geq \frac{-1}{\tan (\frac{\pi}{4b})^2}=-c_3(b).
 \end{equation}
 
 Using (\ref{equation 5}), (\ref{equation 6}), (\ref{equation 7}), (\ref{equation 8}) and (\ref{equation 9})  we have 
 \begin{align}\label{equation 10}
 \sum &\geq \sum_{l\in P_1}\left(\frac {c_1(b)p^2}{l^2}-\frac{c_2(b)|p|}{|l|}\right)+\sum_{l\in P_2}\frac{c_1(b)p^2}{l^2} + \sum_{l\in P_4}(-c_3(b)) \\ &\geq \sum_{l\in P_1\cup P_2}\left(\frac {c_1(b)p^2}{l^2}-\frac{c_2(b)|p|}{|l|}\right)+\sum_{l\in P_4}(-c_3(b)).
 \end{align}
 Note that
 \begin{equation}\label{equation 12}
 \frac{c_1(b)p^2}{l^2}-\frac{c_2(b)|p|}{|l|}\geq -\frac{(c_2(b))^2}{4(c_1(b))^2}=-c_4(b).
 \end{equation}
  If $\frac{|p|}{|l|}\geq \frac{2c_2(b)}{c_1(b)}=c_5(b)$ then 
  \begin{equation}\label{equation 13}
  \frac {c_1(b)p^2}{l^2}-\frac{c_2(b)|p|}{|l|} \geq \frac{c_1(b)p^2}{2l^2}.
  \end{equation}
  
  From Polya-Vinogradov inequality \cite{E}
  \begin{equation*}
  \left|\left\{0<k \leq 2t\sqrt{p}\log p,~k\in L_r\right\}\right|>\sqrt{p}\log p.
  \end{equation*}
  If \begin{equation*} 
  p\geq 4(c_5(b))^2t^2 (\log p)^2=Ct^2(\log p)^2
  \end{equation*}
   and $0<l\leq 2t \sqrt{p}\log p$ then $\frac{|p|}{|l|}\geq c_5(b)$ and  from (\ref{equation 13})  
   \begin{equation}\label{equation 14}
    \frac {c_1(b)p^2}{l^2}-\frac{c_2(b)|p|}{|l|} \geq \frac{c_1(b)p^2}{2l^2}.
    \end{equation}
    Hence if $p\geq 4(c_5(b))^2t^2 (\log p)^2$ and $\frac{p}{2t\sqrt{p}\log p}\leq \frac{p}{4b}$ then from (\ref{equation 10}), (\ref{equation 12}) and (\ref{equation 14})  we have
    \begin{align*}
    \sum &\geq \sum_{\substack{l\in L\\|l|\leq \frac{p}{4b}}}\left(\frac {c_1(b)p^2}{l^2}-\frac{c_2(b)|p|}{|l|}\right)+\sum_{l\in P_4}(-c_3(b))\\ &\geq \sum_{\substack{l\in L\\0<l\leq 2t\sqrt{p}\log p}}\left(\frac {c_1(b)p^2}{l^2}-\frac{c_2(b)|p|}{|l|}\right)+\sum_{\substack{l\in L\\l>2t\sqrt{p}\log p}}\left(\frac {c_1(b)p^2}{l^2}-\frac{c_2(b)|p|}{|l|}\right)+\sum_{l\in P_4}(-c_3(b))\\& 
    \geq \sum_{\substack{l\in L\\0<l\leq 2t\sqrt{p}\log p}}\frac{c_1(b) p^2}{2l^2}+\sum_{\substack{l\in L\\\frac{|l|}{p}\leq \frac{1}{4b}}}(-c_4(b))+\sum_{l\in P_4}(-c_3(b)) 
    \\&\geq \frac{c_1(b)p^{\frac{3}{2}}}{8t^2\log p}-c_4(b)s-c_3(b)s=\frac{ C_1(b)p^{\frac{3}{2}}}{t^2 \log p} -C_2(b)s.
    \end{align*}
\end{proof}

\bibliographystyle{amsplain}

  \end{document}